\documentclass[11pt]{article}
\usepackage{amsmath}
\usepackage{amsthm,amsfonts,amssymb,epsfig,graphics}
\usepackage{pdfsync}
\usepackage[latin1]{inputenc}
\usepackage[english]{babel}
\usepackage{hyperref}
\usepackage{pstricks,pst-plot}

\def\homeo{hom\'eomorphisme}

\def\clos{\mathrm{Clos}}

\def\inte{\mathrm{Int}}

\newcommand{\bbR}{{\mathbb{R}}}

\newcommand{\bbZ}{{\mathbb{Z}}}
\newcommand{\bbD}{{\mathbb{D}}}
\newcommand{\bbF}{{\mathbb{F}}}

\newcommand{\bbS}{{\mathbb{S}}}

\def\fix{\mathrm{Fix}}

\def\homeo{\mathrm{Homeo}}

\def\?{$^{***}$\marginpar{?}}

\newtheorem{theo}{Theorem}

\newtheorem*{ques*}{Question}
\newtheorem*{prop*}{Proposition}
\newtheorem*{conj*}{Conjecture}
\newtheorem*{theo*}{Theorem}
\newtheorem{coro}{Corollary}[section]

\newtheorem{affi*}{Affirmation}
\newtheorem{prop}[coro]{Proposition}
\newtheorem{lemm}[coro]{Lemma}

\newtheorem*{lemm*}{Lemma}
\def\?{\footnote{?}}

\newlength{\espaceavantspecialthm}
\newlength{\espaceapresspecialthm}
\newenvironment{defi}[1][]{\refstepcounter{coro} 
\vskip \espaceavantspecialthm \noindent \textbf{D\'efinition~\thecoro
#1.} }%
{\vskip \espaceapresspecialthm}

{\vskip \espaceapresspecialthm}

\title{Free planar actions of the Klein bottle group}
\author{Fr\'ed\'eric Le Roux\footnote{Laboratoire de math\'ematique (CNRS UMR 8628),  Universit\'e Paris Sud,
91405 Orsay Cedex, France.}
}

\begin{document}
\sloppy 

\maketitle
\begin{abstract}
We describe the structure of the free actions of the fundamental group of the Klein bottle $<a,b \mid aba^{-1}=b^{-1}>$ by orientation preserving homeomorphisms of the plane. The main result is that $a$ must act properly discontinuously, while $b$ cannot act properly discontinuously. As a corollary, we describe some torsion free groups that cannot act freely on the plane. We also find some properties which are reminiscent of Brouwer theory for the group $\bbZ$, in particular that every free action is virtually wandering.
\end{abstract}



\section{Introduction}
It is a natural problem to try to describe 
the finitely generated subgroups of $\mathrm{Homeo}_{0}(\bbR^2)$, the group of orientation preserving homeomorphisms of the plane.
One could also wish to impose some specific property, asking what are the groups that act freely, transitively, minimaly, and so on. To the knowledge of the author, there are very few results, even partial ones, on these very general questions. 

Remember that a group is said to act \emph{freely} if the only element having some fixed point is the trivial element.
A group cannot act freely on the plane if it has some torsion element, because any torsion element of $\mathrm{Homeo}_{0}(\bbR^2)$ is conjugate to a rotation, a theorem of Kerékjártó, see~\cite{ConKol94,Ker19}. What are the other obstructions?

This work is an attempt to give a very partial answer to this question.
 We will describe quite accurately the free planar actions of the group 
$$
BS(1,-1) = <a,b \mid aba^{-1} = b^{-1}>.
$$
 This group is the fundamental group of the Klein bottle, it is also a very special case of the family of Baumslag-Solitar groups (\cite{BauSol62}).
  On the one hand we will see that there are uncountably many different (non conjugated) free actions of $BS(1,-1)$. On the other hand these actions are quite rigid, and share many common features. We will find that there is an analogy between the free actions of $BS(1,-1)$ and the free actions of $\bbZ$ on the plane, as described by Brouwer's \emph{plane translations theorem}. In particular, 
\begin{itemize}
\item  the action of $BS(1,-1)$ is free as soon as the generators $a,b$ have no fixed point;
\item every free action is ``virtually  wandering'' : the action of the  index $2$ abelian subgroup $<a^2,b>$  is wandering.
\end{itemize}

  As a consequence of our study, we will prove that some torsion free groups may not act freely on the plane. 

\bigskip

It is time to describe the simplest free action of $BS(1,-1)$ on the plane\footnote{Of course, $BS(1,-1)$ admits a properly discontinuous action on the plane as the fundamental group of the Klein bottle; but in this text we shall be concerned only with orientation preserving homeomorphisms.}. The dynamics of the generators $a$ and $b$ are described on figure~\ref{fig.simple-action}.
\def\JPicScale{1}
\begin{figure}[htbp]
\begin{center}
\includegraphics{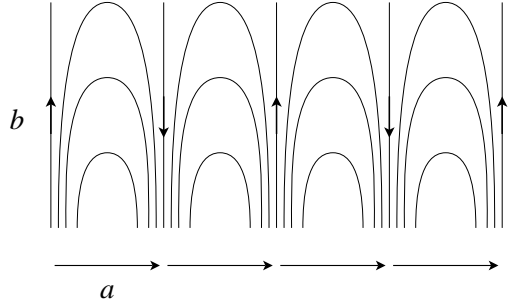}
\caption{generators of the simplest action of $BS(1,-1)$}
\label{fig.simple-action}
\end{center}
\end{figure}
Here is a more accurate description.
Consider the matrices $A,B \in SL(2,\bbR)$,
$$A = \left(
\begin{array}{rr}
 0 & -1 \\
  1  &  0 
\end{array}
\right)
\mbox{ and } B=\left(
\begin{array}{rr}
 2 & 0 \\
  0  &  \frac{1}{2} 
\end{array}
\right).$$
 The relation $ABA^{-1}=B^{-1}$ holds. This gives a (non faithful) action of $BS(1,-1)$ on the plane.
Consider the universal covering map $p: \bbR^2 \to \bbR^2 \setminus\{0\}$ given by $(\theta,r) \to e^{-r+i\theta}$. Then the action of $BS(1,-1)$ on the plane given by $A,B$ lifts under $p$ to a free action on the plane. More precisely, denote by $\widetilde{\mathrm{Homeo}}_{0}(\bbR^2)$ the subgroup of elements of $\mathrm{Homeo}_{0}(\bbR^2)$ that commutes with the map 
$(r,\theta) \to (r, \theta + 2\pi)$. Every element $H \in \widetilde{\mathrm{Homeo}}_{0}(\bbR^2)$ induces an element $P(H)$ of $\mathrm{Homeo}_{0}(\bbR^2)$ that fixes $0$.
The map $P$ is a morphism, and the preimage of the subgroup of $\mathrm{Homeo}_{0}(\bbR^2)$  generated by $A$ and $B$ is isomorphic to $BS(1,-1)$.
The generators are the maps $a,b$ where  $a: (\theta,r) \mapsto (\theta + \pi/2,r)$, and $b$  is the unique lift  of $B$ that satisfies $b(0,r) = (0,r-\log(2))$.

Note that the index $(2p+1)$-subgroup  $<a^{2p+1},b>$ is isomorphic to $BS(1,-1)$, thus we get an infinite family of examples.  The philosophy of our results is that every free action ``looks like'' these models. Theorem~\ref{theo}, Corollaries~\ref{coro.free} and~\ref{coro.wandering}, Lemma~\ref{lem.invariant} and section~\ref{ss.finer} below may be seen as more and more precise descriptions of the free planar actions of $BS(1,-1)$ which illustrate this philosophy.
We will see in particular that for a general free action of $BS(1,-1)$, the action of the generator $a$ is always conjugate to the translation $(r,\theta) \mapsto (r,\theta + \pi)$, while the action of $b$ is never conjugate to a translation (see Theorem~\ref{theo} and the remark that follows).

\bigskip

\paragraph{Aknowledgments}
A motivating problem, asked by Danny Calegari, is to determine whether each finitely generated group, acting on the disk, admits a left-invariant circular order, or equivalently if it is isomorphic to a subgroup of orientation preserving circle homeomorphisms. The answer is affirmative for diffeomorphisms; more generally, it was proved by Danny Calegari that every finitely generated subgroup of $\mathrm{Diffeo}^1_{0}(\bbR^2)$ admitting a compact invariant subset has this property. In the context of this problem, Andrés Navas introduced me with the crystallographic group $G_{1}$ that appears in Corollary~\ref{coro.crystallographic}; it was shortly after that I tried to prove that this group do not act freely on the plane. Andrés also suggested the group $G_{2}$ below.

\section{Results}
\subsection{Main result, easy consequences}

\begin{theo}
\label{theo}
 Assume $a$ and $b$ are fixed point free orientation preserving homeomorphisms of the plane, such that $a b a^{-1} = b^{-1}$.
Then  $a$ is conjugate in $\homeo^+(\bbR^2)$ to a translation. 
\end{theo}
This theorem, which is the main result of the paper, will be proved in section~\ref{sec.proof}.
 It is complemented by the following remark.
\emph{Under the same hypotheses, the map $b$ cannot be conjugate to a translation}. Indeed, let us assume $b = (x,y) \mapsto (x+1,y)$. The map $a$ sends every $b$-orbit to another $b$-orbit, thus it induces an orientation-preserving homeomorphism $\bar a$ of the infinite annulus $\bbR^2/b$. There are only two isotopy classes of orientation preserving homeomorphisms of the annulus, and the relation $aba^{-1}=b^{-1}$ tells us that $\bar a$ is isotopic to the map induced on $\bbR^2/b$ by the rotation $(x,y) \mapsto (-x,-y)$. In particular $\bar a$ extends to a homeomorphism $\bar{\bar{a}}$ of the $2$-sphere $\bbR^2/b \sqcup \{\pm\infty\}$ (the two-ends compactification of the infinite annulus) which exchanges the two points $\pm \infty$.
The homeomorphism $\bar{\bar{a}}$ of the $2$-sphere preserves the orientation, thus it must have a fixed point (this is a  consequence of Lefschetz's formula).
 Thus ${\bar{a}}$ has a fixed point. This means that $a$ preserves some $b$-orbit, let's say $\bbZ \times \{0\}$; the relation $a b a^{-1} = b^{-1}$ yields $a(z+1) = a(z)-1$ for every integer $z$. Either $a$ fixes a point of $\bbZ$, or every point of $\bbZ$ has period two. In this second case $a$ must have a fixed point: this is a consequence of Brouwer theory, see section~\ref{sub.Brouwer} below. In any case the action is not free.

As an easy corollary to the theorem and the remark, we may construct some examples of torsion-free  groups that cannot act freely on the plane.
\begin{coro}
\label{coro.crystallographic}
The following groups are torsion-free and admit no free action by orientation-preserving homeomorphisms of the plane:

\begin{itemize}
\item the crystallographic group $G_{1} = < \alpha,\beta \mid \beta \alpha^2 \beta^{-1}= \alpha^{-2}, \alpha \beta^2 \alpha^{-1}= \beta^{-2}>$;
\item the group $G_{2} =< \alpha,\beta, \gamma \mid \alpha \beta \alpha^{-1}= \beta^{-1},  \beta \gamma \beta^{-1}= \gamma^{-1}>$.
\end{itemize}
\end{coro}
\begin{proof}
The group $G_{1}$ is torsion-free, and the elements $\alpha, \beta$ are non trivial (see~\cite{DerNav10}, paragraph 1.3.1).
If $G_{1}$ acts freely on the plane, then the hypotheses of Theorem~\ref{theo} are satisfied by the maps $a=\alpha$ and $b = \beta^2$.
Thus $\alpha$ is conjugate to a translation. But the maps $a= \beta$ and $b = \alpha^2$ also satisfy the hypotheses of the theorem, and by the above remark we get that $\alpha^2$ is not conjugate to a translation. Since the square of a translation is a translation, this is a contradiction. 

The arguments for the group $G_{2}$ are similar, the only non obvious part is that it is torsion-free and that the elements $\alpha, \beta, \gamma$ are non trivial.
This will follow from the existence of a normal form.
\begin{lemm}
The subgroup $F_{2}$ of $G_{2}$ generated by $\alpha$ and $\gamma$ is free.
Every element of $G_{2}$ has a unique expression of the form $w \beta^n$, with $\omega \in F_{2}$.
Furthermore, the product in $G_{2}$ is given by the formula
$$
(\omega\beta^n).(\omega' \beta^{n'}) = \omega \Phi^n(\omega') \beta^{(-1)^{\sigma(w)}n + n'}
$$
where $\Phi$ is the automorphism of $F_{2}$ that sends $\alpha$ on $\alpha$ and $\gamma$ on $\gamma^{-1}$, and $\sigma: F_{2} \to \bbZ$
is the morphism that counts the sum of the powers of $\alpha$. 
\end{lemm}
\begin{proof}
Consider a word $x$ in $\alpha,\beta,\gamma,\alpha^{-1},\beta^{-1},\gamma^{-1}$. The relations in the group $G_{2}$ can be interpreted by saying that the following operations are valid on the word $x$: the sequence $\beta^\varepsilon \alpha$ may be replaced by $\alpha \beta^{-\varepsilon}$, and the sequence $\beta \gamma^\varepsilon$ may be replaced by $\gamma^{-\varepsilon} \beta$. We may perform these operations on $x$ until every power of $\beta$ has been ``pushed'' to the right of the word. This proves the existence of a normal form in $G_{2}$ as the product of an element of the subgroup $F_{2}$ generated by $\alpha, \gamma$, and a power of $\beta$.

Given such a word $x$, a word $\omega(x)$ in $\alpha,\gamma,\alpha^{-1},\gamma^{-1}$ is constructed as follows. For each occurrence of $\alpha$, we count the number of occurrences of $\beta$ and $\beta^{-1}$ appearing on its left-hand side, and if this number is odd we replace $\alpha$ by its inverse. We do the same for each occurrence of $\alpha^{-1}$; then we delete all the occurrences of $\beta$ and $\beta^{-1}$. The valid operations in $G_{2}$ do not affect the value of $\omega(x)$, thus this construction defines a map $\omega : G_{2} \to \bbF_{2}$. This proves that the subgroup generated by $\alpha$ and $\gamma$ is isomorphic to $\bbF_{2}$, and the uniqueness of the word $\omega \in F_{2}$ in the normal form $\omega \beta ^n$. The uniqueness of the power of $\beta$ is proved similarly. The remaining details, including the formula for the product, is left to the reader.
\end{proof}
The uniqueness of the normal form immediatly entails that $\alpha,\beta,\gamma$ are non trivial elements in $G_{2}$.
It may also be used to check that $G_{2}$ is torsion free.
\end{proof}

The group $G_{1}$ acts faithfully by orientation preserving homeomorphisms on the circle (see section~\ref{sec.line} below), but not on the line (see~\cite{DerNav10}, paragraph 1.3.1). The group $G_{2}$ acts faithfully on the line. This may be seen as follows. It is a classical fact that it is enough to prove that there exists a left-invariant order on $G_{2}$ (see~\cite{DerNav10}, Proposition 1.1.5).
Note that the morphism $\sigma$ defined in the previous Lemma is well defined on $G_{2}$. The function $\eta: G_{2} \to \bbZ$, which is defined using the normal form by $\eta(\omega\beta^n) = n$, restrict to a morphism on the kernel of $\sigma$. The function $w$ defined on $G_{2}$ by $w(\omega\beta^n) = \omega$ restrict to an isomorphism from the kernel of $\eta$ onto the free group $\bbF_{2}$ with two generators. Let us choose an arbitrary left invariant order on $\bbF_{2}$ (it is well known that $\bbF_{2}$ acts faithfully on the line).
In order to define a left-invariant order on $G_{2}$, we first define the set ${\cal P}^+$ of ``positive'' elements, that is, we say that $g > e$ if and only if one of the following holds:
\begin{itemize}
\item $\sigma(g) >0$,
\item $\sigma(g) = 0$ and $\eta(g) >0$,
\item $\sigma(g) = 0$ and $\eta(g) =0$ and $w(g) >e$,
\end{itemize}
where the last inequality refers to the previously chosen order on $\bbF_{2}$.
Then it is easy to check that this set ${\cal P}^+$  is a semi-group, the set ${\cal P}^-$  of inverses of elements of ${\cal P}^+$  is again a semi-group, and that 
$\{ {\cal P}^+ , {\cal P}^-, \{e\} \}$ is a partition of $G_{2}$. Thus this partially defined order extends to a left invariant order on $G_{2}$  (see~\cite{DerNav10}, 1.1.1).

\bigskip

Thus both groups admit faithful (but not free!) actions by orientation preserving homeomorphisms on the plane.
Concerning general actions on the plane, we notice that \emph{every abelian finite subgroup of $\mathrm{Homeo}_{0}(\bbR^2)$ is cyclic}. This may be seen, for example, as an easy consequence of the fact that every compact subgroup of $\mathrm{Homeo}_{0}(\bbR^2)$ is conjugate to a subgroup of $SO(2)$, a theorem of Kerékjártó (\cite{Ker41,Kol06}). Thus, for instance, the Klein four-group do not act faithfully on the plane.
The author is not aware of any other obstruction for a group to act on the plane. In particular,
it would be interesting to find some finitely generated torsion free groups that do not act faithfully on the plane.

\subsection{Analogy with the theory of Brouwer homeomorphisms}
\label{sub.Brouwer}
In this section we show two consequences of Theorem~\ref{theo} that may be regarded as analogous to old results concerning free actions of the group $\bbZ$.

A fixed point free, orientation preserving homeomorphism $h$ of the plane is called a \emph{Brouwer homeomorphism}. The main result of Brouwer theory says that \emph{a Brouwer homeomorphism has no periodic points} (see for example~\cite{LeC06} and the references therein). 
 In other words, the $\bbZ$-action generated by $h$ is free. 
The following result may be seen as an analog of this fact.
\begin{coro}
\label{coro.free}
Let $a,b$ be as in Theorem~\ref{theo}: two Brouwer homeomorphisms satisfying the relation $aba^{-1} = b^{-1}$. Then the action of $BS(1,-1)$ generated by $a$ and $b$ is free.
\end{coro}
 The analogy can be pushed a little further. An action of a group $G$ on a topological space $X$ is said to be \emph{wandering} if every point has a neigbourhood that is disjoint from all its $G$-images. It is part of Brouwer theory that \emph{any free action of the group $\bbZ$ is wandering} (this point of view already appears in~\cite{Thu97}, section~3.5). The same is almost true for the group $BS(1,-1)$.
\begin{coro}
\label{coro.wandering}
Let $a,b$ be as in Theorem~\ref{theo}.
Then the action of the index $2$ abelian subgroup $<a^2,b>$ of $BS(1,-1)$ is wandering. 
More precisely, every disk $D$ such that $b(D) \cap D = \emptyset$ is disjoint from its image under $a^{2p}b^q$ for every $(p,q) \neq (0,0)$.
\end{coro}
Note that the action described on figure~\ref{fig.simple-action} is not wandering, indeed any open set meeting one of the vertical $b$-invariant lines meets its images under $b^{\pm n} a$ for every $n$ large enough (see figure~\ref{fig.non-wandering}).
Thus we cannot dispose of the index $2$ subgroup.

\def\JPicScale{1.3}
\begin{figure}[htbp]
\begin{center}
\includegraphics{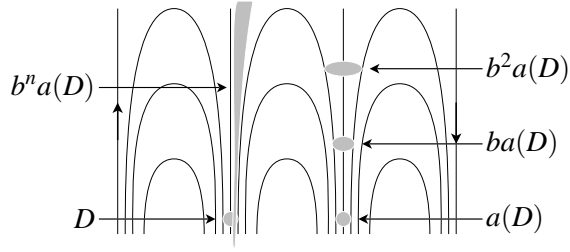}
\caption{This action is not wandering}
\label{fig.non-wandering}
\end{center}
\end{figure}

\bigskip 

In order to prove the corollaries we generalize a definition from~\cite{LeR04}.
Let $\tau$ denotes an affine translation $(x,y) \mapsto (x+\vec v,y)$, and $b$ be any Brouwer homeomorphism. First assume that $b$ commutes with $\tau$.
 Remember that the \emph{index of $b$ along a curve $\gamma : [0,1] \to \bbR^2$} is the real number given by the total angular variation of the vector $\gamma(t) b(\gamma(t))$ when $t$ goes from $0$ to $1$.
Consider a curve $\gamma$ joining some point $x$ to the point $\tau x$. Since $b$ commutes with $\tau$, the vectors $x b(x)$ and $\tau(x) b\tau(x)$ are equal; thus the index of $b$ along $\gamma$ is an integer. The space of curves joining some point to its image under $\tau$ is connected, thus this number does not depend on the choice of the curve $\gamma$, nor on the point $x$; we denote it by $I(b,\tau)$.

Assume now that $\tau b \tau ^{-1} = b^{-1}$, then for any point $x$ the vectors $x b(x)$ and $\tau b(x)  b\tau b(x)$ are opposite. Thus the index of $b$ along a curve joining some point to its image under $ \tau b$ is a half-integer, it does not depend on the choices, again we denote it by $I(b, \tau)$. 
For example, for the action of $BS(1,-1)$ described on Figure~\ref{fig.simple-action}, we have $I(b,a^k) = -\frac{k}{2}$.
Let us summarize this construction.
\begin{defi}
\label{def.index}
If either $b$ commutes or ``anti-commutes'' with an affine translation $\tau$, we have defined a number $I(b,\tau)$ which we call the \emph{index of $b$ relative to $\tau$}. 
\end{defi}

The following lemma generalizes Affirmation 4.15 of~\cite{LeR04}.
\begin{lemm}
\label{lem.invariant}
The number $I(b,\tau)$ is a conjugacy invariant: if $H$ is a homeomorphism that commutes with $\tau$ then $I(HbH^{-1}, \tau) = I(b,\tau)$.
\end{lemm}
\begin{proof}
The space of homeomorphisms that commute with $\tau$ is arcwise connected (this is an easy consequence of Kneser theorem~\cite{Kne26}). 
Whenever $(H_{t})$ is a continuous family 
of homeomorphisms that commute with $\tau$, the number $I_{t} = I(H_{t}bH_{t}^{-1}, \tau)$ is an integer or a half integer that depends continuously on $t$, thus it is constant.
\end{proof}

A \emph{translation domain} for a Brouwer homeomorphism $b$ is a simply connected open subset $O$ of the plane such that $b(O)=O$ and the restriction of $b$ to $O$ is conjugate to a plane translation.
\begin{lemm}
\label{lem.index}
Let $b'$ be a Brouwer homeomorphism that commutes with an affine translation $\tau$.  Assume there exists some translation domain $O$ for $b$ such that $O \cap \tau(O) \neq \emptyset$. Then the index $I(b,\tau)$ is zero.
\end{lemm}

\begin{proof}
We first remark that if $z,z'$ are two points in the plane that are not in the same $\tau$-orbit, then there exists an arc $\gamma$ joining $z$ and $z'$ which is \emph{free} for $\tau$, that is, $\tau(\gamma) \cap \gamma = \emptyset$. Indeed, $\gamma$ may be obtained as the lift of some simple arc in the quotient space $\bbR^2/\tau$ belonging to the right homotopy class.

Under the hypotheses of the lemma, we may find some point $z \in O$ such that $\tau(z)$ belongs to $O$ and is not in the $b$-orbit of $z$.
Since $z,\tau(z)$ belong to the same translation domain, the previous remark provides an arc $\gamma$ joining $z$ to $z'  = \tau(z)$ which is free for $b$, that is, $\gamma \cap b(\gamma) = \emptyset$. Now the Lemma follows from Lemma 4.42 in~\cite{LeR04}.
\end{proof}

\begin{proof}[Proof of corollary~\ref{coro.free} and~\ref{coro.wandering}] 
Let $a,b$ be Brouwer homeomorphisms such that $aba^{-1} = b^{-1}$. We apply Theorem~\ref{theo}: up to conjugacy, we may assume that $a = (x,y) \mapsto (x+1,y)$. 

We first prove Corollary~\ref{coro.wandering}. 
Consider some element $a^{2p}b^q$ with $(p,q)\neq(0,0)$. We first notice that the number $I(b,a)$ is (by definition) a half-integer; and thus the index $I(b,a^{2p}) = 2pI(b,a)$ is non null.
On the other hand, consider any topological disk  $D$ such that $b(D) \cap D = \emptyset$. Then $b^q(D) \cap D = \emptyset$ for every $q \neq 0$, which solves the case $p=0$ (see for example~\cite{Gui94}). Now assume $p \neq 0$.
According to Brouwer theory, there exists a translation domain $O$ that contains $D$, and thus also $b^q(D)$ (see for example~\cite{LeR99}, Theorem~11). Since the index $I(b,a^{2p})$ is not $0$, Lemma~\ref{lem.index} tells us that $a^{2p}(O) \cap O = \emptyset$. \emph{A fortiori} we get $a^{2p}b^q(D) \cap D = \emptyset$. 

Let us prove Corollary~\ref{coro.free}.
Consider some element having normal form $a^p b^q$ with $(p,q) \in \bbZ^2 \setminus\{(0,0\}$ (see Lemma~\ref{lem.normal-form} below). We want to prove that this element is fixed point free.
If $p$ is even then the result follows from Corollary~\ref{coro.wandering}.
 Assume that $p$ is odd. Then we have the identity $(a^p b^q )^2= a^{2p}$.
 Thus  if $a^p b^q$ has a fixed point, this point is also fixed for $a^{2p}$. Since a Brouwer homeomorphism has no periodic point, this is a contradiction.
 \end{proof}

\subsection{A finer description}
\label{ss.finer}
 In the introduction we have claimed on the one hand that there exist uncountably many free actions of $BS(1,-1)$ on the plane, and on the other hand that they are quite rigid.   We would like to provide some evidence to support these claims.
For this we will sketch a finer description of these actions. Here we adopt a more relaxed style ; details are left to the reader.

Consider a free action  $BS(1,-1) \to \homeo_{0}(\bbR^2)$. According to Theorem~\ref{theo}, $a$ is conjugate to a translation. Thus the quotient space $\bbR^2/a^2$ is homeomorphic to an annulus. The map $a$ induces on $\bbR^2/a^2$ an order $2$ homeomorphism $\hat a$, conjugate to the order $2$ rotation of the annulus. The map $b$ induces a homeomorphism $\hat b$ of $\bbR^2/a^2$ that anti-commutes with $\hat a$.
We consider the compactified annulus  $\bbR^2/a^2 \cup \{\pm\infty\}$, which is homeomorphic to the $2$-sphere. Then $\hat b$ extends continuously to a map of the $2$-sphere that fixes both points $\pm \infty$, and has no other fixed point.
\begin{lemm}
The fixed point index of the fixed points $\pm \infty$ for $\hat b$ are even. 
\end{lemm}
\begin{proof}
The argument is essentially the one that we used just before definition~\ref{def.index} to show that $I(b,\tau)$ is a half-integer.
We do the computation inside a chart of $\bbR^2/a^2 \cup \{+\infty\} \simeq \bbR^2$ that sends $+\infty$ to $0$, and in which the map $\hat a$ becomes the linear rotation $(x,y) \to -(x,y)$. In this chart, the vectors from any point $x$ to $\hat b x$ and from $\hat a \hat b (x)$ to $\hat b \hat a \hat b (x) = \hat a x$ are equal. We consider some curve $\gamma$ in $\bbR^2$ joining some point $x$ to $ab(x)$, and the curve $\gamma'$ obtained as the concatenation of $\gamma$ and $ab(\gamma)$. This curve projects down to a curve in $\bbR^2/a^2$ that goes once around $+\infty$, and the index of $\hat b$ along this curve is twice an integer.
\end{proof}
Since the indices are not equal to one, we may apply the results of~\cite{LeR10}, which provides a partial description of the dynamics of $\hat b$ in terms of \emph{Reeb components}. Going back to the free action of $BS(1,-1)$ on the plane, we get the following rough description. Let $m$ be the least number of translation domains for $b$ that are needed to connect a point to its image under $a$.
There exists a unique sequence of Reeb components $(F_{i}, G_{i})_{i \in \bbZ}$ for $b$ such that, for every $i$, $(a(F_{i}), a(G_{i})) =  (F_{i+m}, G_{i+m})$ (see~\cite{LeR05} for the definition of the Reeb components in this context).  
The plane, equipped with the \emph{translation distance} that counts the number of translation domains needed to connect two points (\cite{LeR05}, section~3.1), is quasi-isometric to a tree (this fact uses  Proposition~3.4 of~\cite{LeR10}). In fact, a  planar countable simplicial tree may be naturally  constructed, on which the map $a$ acts as a hyperbolic transformation.

Conversely, let $a$ denotes a plane translation, $\hat a$ the order two map induced by $a$  on the annulus $\bbR^2/a^2$. Let $\beta$ be any fixed point free homeomorphism of the annulus that preserves the orientation and anti-commutes with $\hat a$. Then $\beta$ extends to a homeomorphism of the $2$-sphere that fixes the ends $\pm \infty$. Assume that the index of $+\infty$ is not equal to one\footnote{The author does not now if this is a consequence of the other hypotheses.}.
According to~\cite{LeR04} (section 4.1.e), there exists a unique homeomorphism $b$ of $\bbR^2$, which is a lift of $\beta$, and such that the index $I(b,a^2)$ is not zero; $b$ is called the \emph{canonical lift} of $\beta$.
The relation $\hat a \beta \hat a = \beta^{-1}$ implies that $aba^{-1}$ is a lift of $\beta^{-1}$ ; since $I(aba^{-1},a^2) = I(b, a^2)$ is not zero, this lift is the canonical lift of $\beta^{-1}$, and thus it coincides with $b^{-1}$. In other words, $b$ anti-commutes with $a$, and the homeomorphisms $a,b$ induce a free action of $BS(1,-1)$ on the plane. The image of this action is equal to the group of all the lifts of all the powers of $\beta$.
It is not difficult to construct uncountably many such actions, and in particular to see that every planar countable simplicial  tree, equipped with a hyperbolic map, is realized by a free action of $BS(1,-1)$ on the plane, in the sense of the previous paragraph.

\section{Preliminaries for the proof of Theorem~\ref{theo}}
\label{sec.preliminaires}
\subsection{Basic algebra of $BS(1,-1)$}
Here are some basic facts about the group $BS(1,-1) = < a,b \mid aba^{-1} = b^{-1}>$.
The index $2$ subgroup $<a^2,b>$ is isomorphic to $\bbZ^2$. The index $2$ subgroup $<a,b^2>$ is isomorphic to $BS(1,-1)$.
For any action of $BS(1,-1)$, $a$ preserves the fixed point set of $b$.
\begin{lemm}
\label{lem.normal-form}
Every element of $BS(1,-1)$ is equal to a unique element $a^p b^q$, with $p,q \in \bbZ$.
\end{lemm}
\begin{proof}
 Existence is easy, and the uniqueness is proved by considering some specific faithful action of $BS(1,-1)$, for instance the one described on the first figure of this paper (one could also use the action of $BS(1,-1) $ on the plane as covering automorphisms of the Klein bottle).
\end{proof}

\subsection{One-dimensional actions of $BS(1,-1)$}\label{sec.line}

Figure~\ref{fig.dimension-1} below shows the easiest non trivial action of $BS(1,-1)$ on the line. The generators are obtained as lifts of the action of the matrices $A,B$ on the projective line $P^1(\bbR) \simeq \bbS^1$.

\bigskip
\begin{figure}[htbp]
\begin{center}
\includegraphics{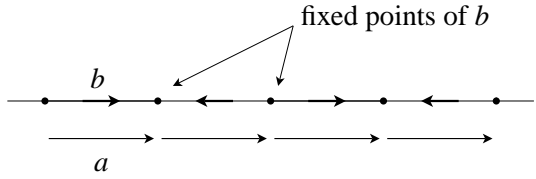}
\caption{The simplest one-dimensional action}
\label{fig.dimension-1}
\end{center}
\end{figure}

\begin{lemm}[actions on the circle]
\label{lem.dimension-one}
Let $a,b$ be two orientation preserving circle homeomorphisms such that $aba^{-1} = b^{-1}$. Assume both $a$ and $b$ have fixed points.
 Then $\fix(a)$ is strictly included in $\fix (b)$. More precisely,
every component of $\bbS^1 \setminus \fix(a)$ contains some points that are fixed by $b$.
\end{lemm}

\begin{proof}
We first find a common fixed point.
Let $y$ be some point of the circle that is fixed by $b$. Then every point $a^n y$ is also fixed by $b$. If $a$ has some fixed point, then this sequence converges to a fixed point of $a$ which is also a fixed point of $b$.

We denote the common fixed point by $\infty$, and we will use that $a$ and $b$ preserves the order on $\bbS^1 \setminus\{\infty\}$ (that is, we consider $a$ and $b$ as homeomorphisms of the line).
Let $x$ be a fixed point for $a$, we argue by contradiction to prove that it is also a fixed point for $b$.
Assume $b(x) <x$. Then $x < b^{-1} x$ since $b$ preserves the orientation. But $a^{-1}x = x$, thus $ b a^{-1}x =  b x < x$ and
$ a b a^{-1}x  < a(x)=x$. This contradicts $a b a^{-1} = b^{-1}$.

Now let $\Delta$ be a component of $\bbS^1 \setminus \fix(a)$. Clearly it is homeomorphic to the line, and invariant under $a$ and $b$. If $b$ has no fixed point on $\Delta$, then $b$ and $ab^{-1}a^{-1}$ pushes points in opposite direction, which is impossible because they are equal. Thus $b$ has some fixed point in $\Delta$ (and actually it must have infinitely many).
\end{proof}

As a corollary of this Lemma, we see that the crystallographic group $G_{1}$ defined in Corollary~\ref{coro.crystallographic} may not act by orientation preserving homeomorphisms of the circle with the generators $a,b$ having rotation number $0$.
Here is a construction of an action where $b$ has rotation number $\frac{1}{2}$. We see $\bbS^1$ as the union of two copies of the closed  interval $[-\infty, +\infty]$ with the two copies of $-\infty$ identified, as well as the two copies of $+\infty$ (we orient the first copy of our interval positively and the second copy negatively). The map $a$ is the translation $x \mapsto x+1$ on each copy of $[-\infty, +\infty]$.
Let $R$ be the map sending the point $x$ of each copy to the point $-x$ in the other copy; this is an order two orientation preserving homeomorphism. Let $b'$ be some orientation preserving homeomorphism of $[-\infty, +\infty]$ such that $ab'a^{-1} = b'^{-1}$, as the '$b$' in Figure~\ref{fig.dimension-1}, and that commutes with the map $x \mapsto -x$. We define $b$ as follows: on the first copy of $[-\infty, +\infty]$ it coincides with $R$, on the second copy it is equal to $b' R$. Then the relations $ab^2 a^{-1}= b^{-2}$ and  
$b a^2 b^{-1}= a^{-2}$ may be easily checked, thus a circle action of the group $G_{1}$ is defined.
The faithfulness of this action may be checked first in restriction to the index four abelian subgroup generated by $a^2,b^2,(ab)^2$; given the absence of torsion element, this entails the faithfulness of the whole action.

\subsection{Limit sets of Brouwer homeomorphisms}
\label{subsec.singular}
The main tool in the proof of Theorem~\ref{theo} will be the singular set of a Brouwer homeomorphism, (probably introduced for the first time in~\cite{HomTer53}).
For the proofs we refer the reader to the section~5 in~\cite{LeR05} (which may be read independently of the remaining of~\cite{LeR05}).

In this section we consider a single Brouwer homeomorphism $a$. A set $E$ is said to be \emph{free} if $a(E) \cap E = \emptyset$.
For every free topological closed disk $D$, the sequence $(a^n(D))_{n \geq 0}$ is made of pairwise disjoint sets.
Furthermore, it converges in the space of compact subsets of the sphere $\bbR^2 \cup \{\infty\}$, equipped with the Hausdorff topology. The limit, or more precisely its intersection with the plane, is called the \emph{positive limit set} of $D$ and denoted by $\lim^+ D$. It is disjoint from $D$ and from all its iterates $a^n(D)$.  For every point $x$, we define the \emph{positive limit set} of $x$ as
$$
 \lim{}^+ x = \bigcap\{ \lim{}^+ D, x \in \inte(D)\}.
$$
 The set $\lim^-x$ is defined as the positive limit set of $x$  for the homeomorphism $a^{-1}$. We will use the notation $\lim_{a}^+$ when we need to emphasize that $a$ is the homeomorphism that is used.
This construction has the following properties.

\begin{enumerate}
\item We have  $y \in \lim^+x \Leftrightarrow x \in \lim^-y$ and this holds if and only  if there exists a sequence $(z_{n})_{ n \geq 0}$ of points converging to $x$ such that the sequence 
$(a^n z_{n})_{ n \geq 0}$ converges to $y$ (then we say that the couple $(x,y)$ is \emph{singular} for $a$).

\item  The map $a$ is conjugate to a translation if and only if it admits no singular couple.

\item For every $x$, $\lim_{a^2}^+x = \lim_{a}^+x$.

\item The sets $\lim^+x$ and $\lim^-x$ are disjoint and do not contain $x$. The sets  $(\lim^+x )\cup \{\infty\}$ and $(\lim^-x ) \cup \{\infty\}$ are compact connected subsets of $\bbR^2 \cup\{\infty\}$. They are invariant under $a$.

\item These are conjugacy invariants: for every homeomorphism $\Phi$ of the plane, 
$\lim_{\Phi a \Phi^{-1}}^\pm \Phi(x) = \Phi(\lim_{a}^\pm x)$.
\end{enumerate}

If $k$ is a free connected compact subset of the plane, then again the sets in the sequence $(a^n(k))_{n \geq 0}$ are pairwise disjoint, and the sequence converges. The proof is similar to the proof of convergence for disks (see Lemme~5.1 in~\cite{LeR05}), and left to the reader.
Thus we may define\footnote{Note that $\lim_{a}^+x$ is in general not equal to $\lim_{a}^+\{x\}$, since the latter is always empty.}
 the limit set $\lim^+(k)$. 
 The sets $\lim^-(k)$ and $\lim^+(k)$ are disjoint and disjoint from $k$.

We will denote by $V_{a}^+(x)$  the connected component of $\bbR^2 \setminus \lim_{a}^+(x)$ that contains $x$. It is invariant under $a$ (see Lemme~5.8 in~\cite{LeR05}).
Likewise, for any free connected compact set $k$, the connected component $V_{a}^+(k)$ of $\bbR^2 \setminus \lim_{a}^+(k)$ that contains $k$ is invariant under $a$ (the proof is similar to the proof of Lemme~5.4 in~\cite{LeR05}). We denote $V_{a}^+(k)$ by $V^+(k)$ when there is only one Brouwer homeomorphism
 under consideration. 
  The following result generalizes Proposition 5.5 in~\cite{LeR05}, where $k^-$ was a closed disk and $k^+$ was an arc.
\begin{prop}
\label{prop.generalise55}
Let $k^-, k^+$ be two compact connected free sets. Assume that $k^+$  meets both $\lim^+k^-$ and $V^+(k^-)$.
Then there exists $n_{0}>0$ such that for every $n \geq n_{0}$, $a^n k^- \cap k^+ \neq \emptyset$.
\end{prop}
 
The following construction will be useful for the proof. Denote by $\mathrm{Full}(k)$ the union of $k$ and of all the bounded connected components of $\bbR^2 \setminus k$. If $k$ is compact and connected then 
$\mathrm{Full}(k)$ is compact and connected. It coincides with the intersection of all topological disks containing $k$, and there exists a decreasing sequence of topological disks whose intersection is equal to $\mathrm{Full}(k)$.
 If in addition $k$ is free then $\mathrm{Full}(k)$ is free, essentially because all the orbits of $a$ tends to infinity, which prevents $a(k)$ to be included in $\mathrm{Full}(k)$.
Thus the set $\lim^+ \mathrm{Full}(k)$ is also defined. Furthermore, it is easy to see that it is equal to $\lim^+ k$.

%

\begin{proof}
We generalize the proof of~\cite{LeR05}. We will need the following simple consequence of Franks's lemma: \emph{whenever $D_{1}, D_{2}$ are two free disks, the set of times $n$ such that $a^n(D_{1})$ meets $D_{2}$ is an interval of $\bbZ$} (see Lemme~7 in~\cite{LeR99}).

Let $k^-, k^+$ be as in the proposition. We begin with an easy case, namely assuming that $k^+$ meets  some iterate $a^{n_{0}}(k^-)$. In this case, let us prove that $k^+$ must also meet all the iterates $a^n (k^-)$ with $n \geq n_{0}$. Assume on the contrary that $a^{n_{1}} (k^-) \cap k^+ = \emptyset$ for some $n_{1}> n_{0}$. 
First note that $k^+$ is included in the unbounded connected component of $\bbR^2 \setminus a^{n_{1}}( k^-)$. Indeed it is disjoint from $a^{n_{1}} (k^-)$ and meets the set $\lim_{}^+k^-$, which is disjoint from $a^{n_{1}} (k^-)$ and unbounded. Likewise $ a^{n_{1}} (k^-)$  is included in the unbounded connected component of $\bbR^2 \setminus k^+$. For otherwise we would have $\lim_{}^+k^- \subset \lim_{}^+\mathrm{Full}(k^+) =  \lim_{}^+ k^+$ and the set $k^+$ would meet $\lim_{}^+k^+$, which cannot happen.

Thus there exist some free open disks $D^-, D^+$ containing respectively $k^-, k^+$  and such that $a^{n_{1}} (D^-) \cap D^+ = \emptyset$. Since $k^+$ meets $\lim_{}^+k^-$ there exists $n_{2}> n_{1}$ such that $a^{n_{2}} (D^-) \cap D^+ \neq \emptyset$. 
In this situation $D^+$ meets $a^{n_{0}}(D^-)$ and $a^{n_{2}}(D^-)$ but not $a^{n_{1}}(D^-)$, which contradicts Franks's lemma.

We now face the general case. Again we consider some free open disk  $D^+$ containing $k^+$. According to Schoenflies theorem, there exists a homeomorphism of the plane that sends $D^+$ to a euclidean disk $B_{x_{0}}$ whose center we denote by $x_{0}$. We may further require that some point of $k^+ \cap V_{}^+(k^-)$ is sent to $x_{0}$. Conjugating $a$ by this change of coordinate, we may assume that $B_{x_{0}}$ is a free euclidean open disk containing $k^+$ whose center $x_{0}$ is a point of $k^+ \cap V_{}^+(k^-)$.

Consider a point $x$ of $V_{}^+(k^-)$. We say that $x$ is a \emph{neighbour} if there exists a euclidean free open disk $B_{x}$ (called a \emph{neighbour disk}), centered at $x$, meeting  $\lim_{}^+k^-$. The set of neighbours is open and contains the point $x_{0}$. Assume $x$ is a neighbour. Then exactly one of the following holds:

\begin{enumerate}
\item There exists a compact connected 
 set $k^g_{x}$, included in a neighbour disk $B_{x}$ and disjoint from $\lim_{}^+ k^-$, containing $x$ and meeting some iterate of $k^-$ (then the point $x$ will be called \emph{good}; note that in this case $k^g_{x}$ meets only a finite number of iterates of $k^-$).
\item There exists a compact connected 
 set $k^b_{x}$, included in a neighbour disk  $B_{x}$ and meeting $\lim_{}^+ k^-$, containing $x$ and disjoint from all the iterates of $k^-$ (then the point $x$ will be called \emph{bad}).
\end{enumerate}
Indeed we may choose a segment $\gamma$ inside  a neighbour disk $B_{x}$ with one end-point equal to  $x$ and the other on $\lim_{}^+ k^-$, and otherwise disjoint from 
$\lim_{}^+ k^-$. If $\gamma$ does not meet any iterate of $k^-$ we are in the bad case, otherwise we are in the good case. Now assume both cases occur simultaneously.
Then $k'_{2} = k^g_{x} \cup k^b_{x}$ is a compact connected free set meeting $\lim_{}^+ k^-$ and some iterate of $D^+$, but not infinitely many of them, in contradiction to the easy case.

 We will actually prove that \emph{every neighbour point is good}. From this we deduce that the point $x_{0}$ is good, thus $k^+$, being included in $B_{x_{0}}$ and meeting the set $\lim^+ k^-$,  must meet some iterate of $k^-$. Then we are back to the easy case of the proposition, and the proof is complete.

Let us prove that every neighbour point is good. It is easy to see that the set of good points is open. Since both types are exclusive, a bad point has a neighbourhood which does not meet any iterate of $k^-$, and from this we see that the  set of bad points is also open. Denote by $N$ the set of neighbour points, and let $N'$ be some connected component of $N$. It is enough to prove that $N'$ contains some good points.
Let $z$ be any point which is simultaneously on the boundary of $N'$ and on the boundary of $V_{}^+(k^-)$. The situation is now similar to that of  Affirmation 5.10 of~\cite{LeR05}, and the same proof shows that $N' \cup\{z\}$ is a neighbourhood of $z$ in the space $V^+(k^-) \cup \{z\}$ (the key argument here is Alexander's lemma, see~\cite{LeR05}). Being on $\partial V^+(k^-) \subset \lim^+ k^-$, the point $z$ is accumulated by iterates $a^n(k^-)$, which are all included in $ V_{}^+(k^-)$ since this last set is invariant under $a$. In particular $N'$ meets some of these iterates. Thus it contains some good points.
\end{proof}


\paragraph{Homeomorphisms of the disk}
We will need to adapt the above definitions to a slightly different setting. Let $a$ be an orientation preserving homeomorphism of the closed two-disk $\bbD^2$, which has no fixed point in the interior. The interior of $\bbD^2$ is homeomorphic to the plane, thus the above definitions make sense for $a_{\mid \inte(\bbD^2)}$, but we also want to consider singular couples $(x,y)$ where $x$ or $y$ (or both) belongs to the boundary, and to define the limit set of a non fixed boundary point. We proceed as follows. We identify $\bbD^2$ with the unit disk in the plane. Let $h$ be the circle homeomorphism such that, in polar coordinates,  $a(1,\theta) = (1, h(\theta))$ for every $\theta \in \bbS^1$. Extend $a$ to a homeomorphism $\bar a$ of the plane by setting $a(r, \theta) = (r, h(\theta))$ for every $\theta \in \bbS^1$ and $r>1$. Let $O$ be the complement of the fixed points set of $\bar a$ in the plane. It is easy to check that $O$ is simply connected. Thus we may identify $O$ with the plane, and consider $a' = \bar a_{\mid O}$ as a Brouwer homeomorphism.
In particular the sets $\lim_{a'}^+ x$, $\lim_{a'}^+ k$ are defined for every point $x$ and every free compact connected set $k$ in $O$.
Now we set
$$\lim{}_{a}^+ x = \lim{}_{a'}^+ x, \ \ \ \lim{}_{a}^+ k = \lim{}_{a'}^+ k$$
for any non fixed point $x$ or free compact connected set $k$ in $\bbD^2$. It is easy to see that these are closed subsets of $\bbD^2 \setminus \mathrm{Fix}(a)$.
We have the following characterizations for every point $y \in \bbD^2$ which is not fixed by $a$:
\begin{itemize}
\item $y \in \lim_{a}^+ x$ if and only if there exists a sequence $(z_{n})_{ n \geq 0}$ of points of $\bbD^2$ converging to $x$ such that the sequence 
$(a^n z_{n})_{ n \geq 0}$ converges to $y$.
\item $y \in \lim_{a}^+ k$ if and only if there exists a sequence $(z_{n})_{ n \geq 0}$ of points of $k$ such that the sequence 
$(a^n z_{n})_{ n \geq 0}$ converges to $y$.
\end{itemize}
We also define the set $V_{a}^+(k) = V_{a'}^+ (k) \cap \bbD^2$; the negative limit sets are defined similarly.

\subsection{Prime ends for limit sets of Brouwer homeomorphisms}

We briefly review the theory of prime ends compactification, see~\cite{Mat82} for details. Let $U$ be a simply connected open set in the plane which is not equal to the whole plane. 
The space  $\hat U$ of prime points of $U$ is a topological space, naturally associated to $U$, with the following properties.  (1) $U$ identifies with a subset of $\hat U$; the points in $\partial \hat U = \hat U \setminus U$ are the \emph{prime ends} of $U$. (2) The pair $(U, \hat U)$ is homeomorphic to the pair $(\inte (\bbD^2), \bbD^2)$.
(3) Any homeomorphism $a : \clos(U) \to \clos(U)$ extends to a unique homeomorphism $\hat a : \hat U \to \hat U$. This extension process is natural, in particular the map $a \mapsto \hat a$ is a group homomorphism. This last point is of course especially crucial for the study of group actions.

Remember (\cite{Mat82}) that an \emph{end-path} in $U$ is a continuous mapping $\gamma : (0,1] \to U$ such that, when $t$ tends to $0$, $\gamma(t)$ converges in $\clos(U)$ to a point of $\partial U$, which is called the limit point of $\gamma$ and denoted by $\lim_{\clos(U)}\gamma$. If $\gamma$ is an end-path then $\lim_{\hat U}\gamma$ is also a single prime end (\cite{Mat82}, Lemma 14.1). Such a prime end is called \emph{accessible}. The set of accessible prime ends is dense in $\partial \hat U$ (\cite{Mat82}, Theorem 17.2).

According to Lemma~14.1 in~\cite{Mat82}, two end-pathes having distinct limit points in $\partial U$ also have distinct limit prime ends in $\partial \hat U$. This entails immediatly the following Lemma.
\begin{lemm}
\label{lemm-fixed accessible prime ends}
Let $a$ be a Brouwer homeomorphism, and $U$ be a non empty simply connected  open set such that $a(U) = U$. 
Assume that $U$ is  not equal to the whole plane.
Then no accessible prime end in $\partial \hat U$ is fixed by $\hat a$. In particular,   the set of fixed points of $\hat a$ is nowhere dense in $\partial \hat U$.
\end{lemm}

We consider a Brouwer homeomorphism $a$. The following lemma is essentially another formulation of Proposition~\ref{prop.generalise55}.
\begin{lemm}
\label{lem.limit-prime-ends}
Let $z$ be a point such that $\lim_{a}^+z$ is non-empty. Let $U = V^+ z$ be the connected component of $\bbR^2 \setminus \lim^+z$ that contains $z$, and $\hat a : \hat U \to \hat U$ denotes the prime ends compactification of $a$ on $U$. Then the set  $\lim_{\hat a}^+ z$ is equal to $\partial \hat U \setminus\mathrm{Fix}(\hat a)$.

Likewise, let $k$ be a compact connected free set such that $\lim_{a}^+k$ is non-empty. Let $U = V^+ k$, and $\hat a : \hat U \to \hat U$ denotes the prime ends compactification of $a$ on $U$. Then the set  $\lim_{\hat a}^+ k$ is equal to $\partial \hat U\setminus\mathrm{Fix}(\hat a)$.
\end{lemm}
\begin{proof}
First note that the set $U$ is invariant under $a$ (\cite{LeR05}, Lemme~5.4).
By definition the set $\lim_{a}^+z$ is disjoint from $U = V_{a}^+ z$. Thus $\lim_{\hat a}^+z \subset \partial \hat U$.
For the reverse inclusion, since the accessible prime ends are dense on $\partial \hat U$ and the set $\lim_{a}^+z$ is closed in $\partial \hat U \setminus\mathrm{Fix}(\hat a)$, 
we just need to show that every accessible prime end $e \in \partial \hat U$ belongs to $\lim_{\hat a}^+(z)$.
Let  $\gamma$ be an end-path in $U$ such that $\lim_{\hat U} \gamma = e$, with $\gamma(1) = z$. According to Proposition 5.5 of~\cite{LeR05}, 
for any disk $D$ in $U$ that contains $z$ in its interior, $\gamma$ meets $a^{n}(D)$ for an infinite number of $n \geq 0$. We deduce that 
the set $\lim_{a}^+z$ meets the closure of $\gamma$ in $\hat U$; since it is disjoint from $U$, it contains $e$.

The second case follows from analogous considerations, replacing Proposition 5.5 of~\cite{LeR05} by Proposition~\ref{prop.generalise55} above.
\end{proof}

\section{Proof of Theorem~\ref{theo}}
\label{sec.proof}

Let us explain the main idea of the proof. We consider two Brouwer homeomorphisms $a,b$ such that $aba^{-1} = b^{-1}$. We assume that $a$ is not conjugate to a translation, and we look for a contradiction. By hypothesis the singular set of $a$ is non empty: there exist two points $x,y$ such that $y \in \lim_{a}^+ (x)$.
The sets $\lim_{a}^+ (x)$ and $\lim_{a}^- (y)$ are one-dimensional closed subsets of the plane. To simplify, let us assume that there exists a simply connected open set $O$ whose boundary is the union of these two limit sets. Then the set $O$ is essentially invariant under $a$ and $b$.
By considering the prime-ends compactification of $O$, we get an action of $BS(1,-1)$ on the disk; we may further symplify the action to get the following situation (figure~\ref{fig.before-after-two}, right): the action is free on the interior of the disk, $a$ has exactly two fixed points $N,S$ on the boundary, $b$ has only isolated fixed points in each component $\Delta,\Delta'$ of $\partial \bbD^2 \setminus\{N,S\}$. We now use the dynamics of $b$. A classical property of local dynamics, that goes back to Birkhoff, allows to find two compact connected sets $k,k'$ that meet respectively $\Delta,\Delta'$, and that are positively invariant under $b$. We now remember that the boundary of $O$ is made up of two mutually singular sets. This entails that almost all $a$-iterates of $k$ will meet $k'$. Consider a point in the intersection  $a^{2n}(k) \cap k'$ (figure~\ref{fig.contradiction}). The set $a^{2n}(k)$ is positively invariant under $b$, thus the $b$-iterates of this point must converge towards a point of $\Delta$, but they must also converge to a point of $\Delta'$. This is a contradiction.

Unfortunately, such a set $O$ do not always exists. But this difficulty may be overcome by using two successive prime-ends compactifications in order to get the disk action.

\bigskip

We now turn to the details of the proof. Let $a,b$ be two Brouwer homeomorphisms satisfying $aba^{-1} = b^{-1}$.

\begin{lemm}\label{lemm.invariance}
For every point $x$, the following  sets are invariant under $b^2$: $\mathrm{lim}^+_{a}(x)$, $\mathrm{lim}^-_{a}(x)$, the connected component $V_{a}^+(x)$ of $\bbR^2 \setminus\lim^+_{a} (x)$,  and the connected component $V_{a}^-(x)$ of $\bbR^2 \setminus\lim^-_{a} (x)$.
\end{lemm}

\begin{proof}
We prove the invariance for $\mathrm{lim}^+_{a}(x)$, the proof is the same for the two other sets.
We use the properties of the limit sets enumerated in section~\ref{subsec.singular}. In particular the limit sets with respect to $a$ and $a^2$ are the same, and the limit sets are conjugacy invariants. Since $b$ commutes with $a^2$, we deduce that $b(\mathrm{lim}^+_{a}(y)) = \mathrm{lim}^+_{a}(b(y))$ for every $y$. Then
$$
\begin{array}{rcl}
 b^2  \left(\mathrm{lim}^+_{a}(x) \right) &  = & b \left(\mathrm{lim}^+_{a}(b(x)) \right)  \\
  &  =  & a b^{-1} a^{-1}  \left(\mathrm{lim}^+_{a}(b(x)) \right)   \\
  &   =  & \mathrm{lim}^+_{a}(x)
\end{array}
$$
where the last equality also uses that $\mathrm{lim}^+_{a}(y)$ is invariant under $a$ for every $y$.
\end{proof}

The Brouwer homeomorphism $b^2$ still satisfies $a(b^2)a^{-1} = (b^2)^{-1}$.
Thus, up to replacing $b$ by $b^2$, we may assume the following property:
\emph{for every point $x$, the four sets appearing in Lemma~\ref{lemm.invariance} are invariant under $b$.}

From now on we argue by contradiction, assuming that $a$ is not conjugate to a translation. According to point 2 in section~\ref{subsec.singular}, there exists two points $x,y$ such that $x \in \lim_{a}^-y$.
In particular, let us consider the set $U_1 = V_{a}^-(y)$.
 This is a simply connected proper open subset of the plane, which is invariant under $a$ and $b$. Let $\hat U_1$ denotes the prime-ends compactification of $U_1$. We denote by $a_{1} = \hat a, b_{1} = \hat b$ the induced homeomorphisms, which still satisfy the relation $a_{1}b_{1}a_{1}^{-1} = b_{1}^{-1}$. The point $y \in U_{1}$ identifies with a point in $\hat U_{1}$ which we denote by $y_{1}$.

\begin{lemm}
The homeomorphisms $a_{1}$ and $b_{1}$ have a common fixed point on $\partial \hat U_1$.
\end{lemm}

\begin{proof}
By the Brouwer fixed point theorem, both $a_{1}$ and $b_{1}$ have some fixed point on $\hat U_{1}$, and since by hypothesis they are fixed point free on $U_{1}$, the fixed points are on the boundary. Thus Lemma~\ref{lem.dimension-one} applies and provides a common fixed point.
\end{proof}

According to Lemma~\ref{lemm-fixed accessible prime ends}, $\partial \hat U_1$ contains some point which is not fixed by $a_{1}$.
Let $\Delta$ be a connected component of $\partial \hat U_1 \setminus \fix(a_{1})$ in $\partial \hat U_1$.
According to Lemma~\ref{lem.dimension-one},  $\Delta$ is invariant under $b_{1}$, and $b_{1}$ has some fixed point on $\Delta$. We would like this fixed point to be isolated among the fixed points of $b_{1}$, and for this we use the following construction.
Let $I$ be any connected component of $\Delta \setminus \fix(b_{1})$. Since $a_{1}$ preserves the fixed point set of $b_{1}$ and has no fixed point on $\Delta$, we have $a_{1} (I) \cap I = \emptyset$. Let $\sim$ be the equivalence relation whose non trivial classes are the connected components of $\partial \hat U_{1} \setminus \{a_{1}^n (I), n \in \bbZ \}$. Denote by $p: \hat U_1  \to \hat U_1/\sim$ the quotient map. It is easy to see that the quotient space $\hat U_1/\sim$ is  again homeomorphic to a disk. The homeomorphisms $a_{1}, b_{1}$ induce homeomorphisms of $\hat U_1/\sim$ which we denote by $a_{2}, b_{2}$, and the relation $a_{2} b_{2} a_{2}^{-1} = b_{2}^{-1}$ is still satisfied. 
The complement of $\Delta$ in $\partial \hat U_{1}$ is sent onto a single point, which is the only fixed point of $a_{2}$, and on $p(\Delta)$ the action of $<a_{2}, b_{2}>$ is conjugate to the easiest line action, as pictured on figure~\ref{fig.dimension-1}.
Also note that according to Lemma~\ref{lem.limit-prime-ends}, the set  $\lim_{a_{1}}^- y_{1}$ is equal to $\partial \hat U_1 \setminus \mathrm{Fix}(a_{1})$, from thus we deduce easily that $\lim_{a_{2}}^- p(y_{1}) = \partial \hat U_1/\sim \setminus \mathrm{Fix}(a_{2}) = p(\Delta)$.

From now on we forget about the initial action and work in $\hat U_1/\sim$, which we identify with $\bbD^2$; we denote the point $p(y_{1})$ by $y_{2}$, and keep the notation $\Delta$ for $p(\Delta)$.  The salient features are the following.
\begin{itemize}
\item The maps $a_{2},b_{2}$ are orientation preserving homeomorphisms of the disk $\bbD^2$ that satisfy $a_{2}b_{2}a_{2}^{-1} = b_{2}^{-1}$ and have no fixed point on $\inte(\bbD^2)$.
\item On $\partial \bbD^2$, the maps $a_{2}$ and $b_{2}$ have a single common fixed point,  which we denote by $\infty$, and on $\Delta = \partial \bbD^2 \setminus\{\infty\}$ the dynamics of $<a_{2}, b_{2}>$ is as pictured on figure~\ref{fig.dimension-1}.
\item There exists a point $y_{2} \in \inte(\bbD^2)$ such that  $\lim_{a_{2}}^- y_{2}  = \partial \bbD^2 \setminus\{\infty\}$.
\end{itemize}

\begin{figure}[htbp]
\begin{center}
\includegraphics[width=4.5cm]{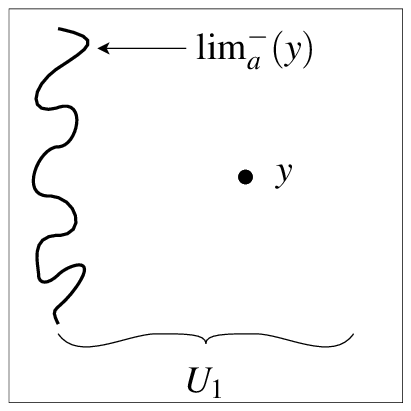}
\hskip1cm
\includegraphics[width=4.5cm]{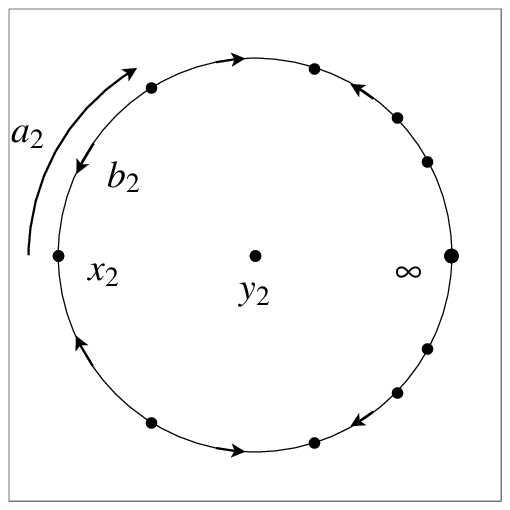}
\caption{Before and after the first prime-ends compactification}
\label{fig.one}
\end{center}
\end{figure}

Let $x_{2}$ denote a fixed point of $b_{2}$ whch is an attractor for the restriction of $b_{2}$ to $\partial \bbD^2$.
Identifying $\bbD^2$ with the unit disk in the plane, we may extend $b_{2}$ to a homeomorphism $b$ of the plane such that $x_{2}$ is an attractor for the restriction of $b$ to $\bbR^2 \setminus \inte(\bbD^2)$. Then we may apply the following lemma with $x=x_{2}$.
\begin{lemm}[Birkhoff's lemma, \cite{Bi20}, paragraph 51]
Let $b$ be a homeomorphism of the plane, and $x$ an isolated fixed point for $b$. Then one of the following holds.
\begin{enumerate}
\item The point $x$ has a basis of connected (closed) neighbourhoods $N$  satisfying $b(N) \subset N$.
\item For any small enough neighbourhood $N$ of $x$, there exists a connected compact set $k'$ included in $N$, containing $x$ and a point of $\partial N$, and satisfying $b^{-1}(k') \subset k'$.
\end{enumerate}
 
\end{lemm}
If the first case of the lemma occurs, then we define $k = N \cap \bbD^2$ with $N$ a small attracting neighbourhood of $x_{2}$ provided by the Lemma.
 In the opposite case, let us consider some small neighbourhood $N$ and the compact set $k'$ given by the second case of the lemma. Note that, since $x_{2}$ is an attractor for the restriction of $b$ to the complement of $\inte(\bbD^2)$, $k'$ is included in $\inte(\bbD^2) \cup \{x_{2}\}$.
Then the set $k=a_{2} (k')$ satisfies $b_{2}(k) \subset k$.
To summarize, in any case \emph{there exists  a compact connected subset  $k$ of $\bbD^2$, included in an arbitrarily small neighbourhood $N$ of the fixed point $x_{2}$ of $b_{2}$, that contains $x_{2}$ and is not included in $\Delta$, and such that $b_{2}(k) \subset k$.} 
Since $a_{2}$ has no fixed point on $\Delta$, we may find a topological closed disk $D \subset \bbD^2$ which is free for $a_{2}^2$, and such that the set 
$$\bigcup_{n \in \bbZ}a_{2}^n(D)$$
contains $\Delta$. In the above construction we choose $N$ small enough so that $k$ is free for $a_{2}$ and included in $D$. Since the set $\lim_{a_{2}}^+ D$ is disjoint from $a_{2}^n(D)$ for every $n$, this implies that
the set
$$
\lim{}_{a_{2}}^+ k = \lim{}_{a_{2}^2}^+ k \subset \lim{}_{a_{2}^2}^+ D 
$$
is disjoint from $\Delta$.

\begin{lemm*}
 The set $\lim_{a_{2}}^+ k$ contains $y_{2}$.
\end{lemm*}

\begin{proof}
Let $D$ be a closed disk containing $y_{2}$ in its interior. The set $\lim_{a_{2}}^- D$ contains $\lim_{a_{2}}^- y_{2} = \partial \bbD^2 \setminus\{\infty\}$, in particular it contains $x_{2}$.
Let $\gamma$ be any arc joining $y_{2}$ to a point of $k$ inside $\inte(\bbD^2)$. The arc $\gamma$ is disjoint from $\lim_{a_{2}}^- y_{2}$, thus it is also disjoint from $\lim_{a_{2}}^- D$ if $D$ is small enough. In this case $\gamma$ is included in the complementary component $V^-_{a_{2}} (D)$ of $\lim_{a_{2}}^- D$ that contains $D$, and in particular this set contains a point of $k$. Since $k$ meets both $\lim_{a_{2}}^- D$ and $V^-_{a_{2}} (D)$ we may apply Proposition~\ref{prop.generalise55} (with $k^-=D$, $k^+=k$, and $a=a_{2}^{-1}$).
We get that there exists $n_{0}$ such that for every $n \geq n_{0}$, the set $a_{2}^{-n}(D)$ meets $k$, in other words $a_{2}^{n}(k)$ meets $D$. Thus $\lim_{a_{2}}^+ k$ meets $D$. Since this happens for every small enough $D$, we get that $y_{2}$ belongs to $\lim_{a_{2}}^+ k$.
\end{proof}

\bigskip
\bigskip

 Let $U_{2}$ be the connected component of $\inte(\bbD^2) \setminus \lim^+_{a_{2}} k$ containing $\inte(\bbD^2) \cap k$. This is a simply-connected proper open subset of the disk. We consider  the prime-ends compactification  $\hat U_{2}$ of $U_{2}$, and we let $a_{3} = \hat a_{2}, b_{3} = \hat b_{2}$ be the induced homeomorphisms. 
Since the set $\lim{}_{a_{2}}^+ k $ is disjoint from $\Delta$,
the points of $\Delta$ are all accessible from $U_{2}$, so $\Delta$ identifies with an open interval  of $\partial \hat U_{2}$. In particular the point $x_{2}$ and the set $k$ identify with a point and a set in $\hat U_{2}$ which we still denote by $x_{2}$ and $k$.
\begin{lemm}
The set $\lim^+_{a_{3}} k$   is equal to 
$$\partial \hat U_{2} \setminus  (\Delta \cup \mathrm{Fix}(a_{3}).$$ 
\end{lemm}
\begin{proof}
As in the proof of Lemma~\ref{lem.limit-prime-ends}, we use Proposition~\ref{prop.generalise55} to show that every accessible prime end on 
$\partial \hat U_{2} \setminus  \clos(\Delta)$ is accumulated by iterates of $k$. Details are left to the reader.
\end{proof}
Since the set $\lim{}_{a_{2}}^+ k $ is not empty, $\hat U_{2} \setminus  \Delta$ is a non-empty closed interval. This interval contains some points that are not fixed under $b_{3}$ (Lemma~\ref{lemm-fixed accessible prime ends}), thus we may consider some connected component $I'$ of 
$$\lim^+_{a_{3}} k \setminus \left(\Delta \cup \fix(b_{3})\right).$$
As above,  we identify points in the same component of 
$$
\partial \hat U_{2} \setminus \left( \bigcup_{n \in \bbZ} a_{3}^n(I') \cup \Delta\right).
$$
The resulting space is again homeomorphic to the disk, and we get the following situation.
\begin{figure}[htbp]
\begin{center}
\includegraphics[width=5.5cm]{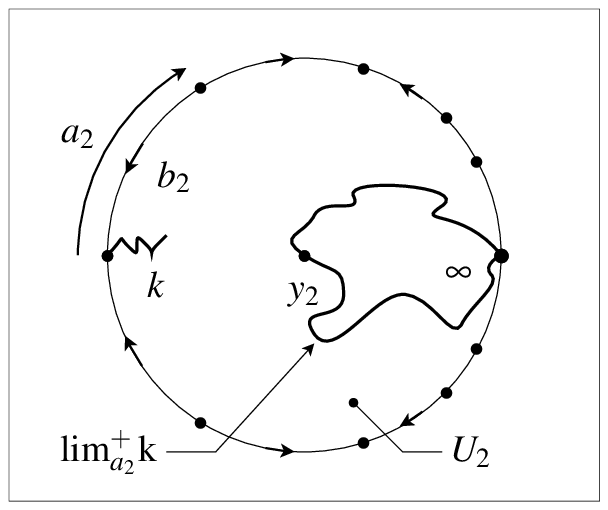}
\hskip1cm
\includegraphics[width=5.5cm]{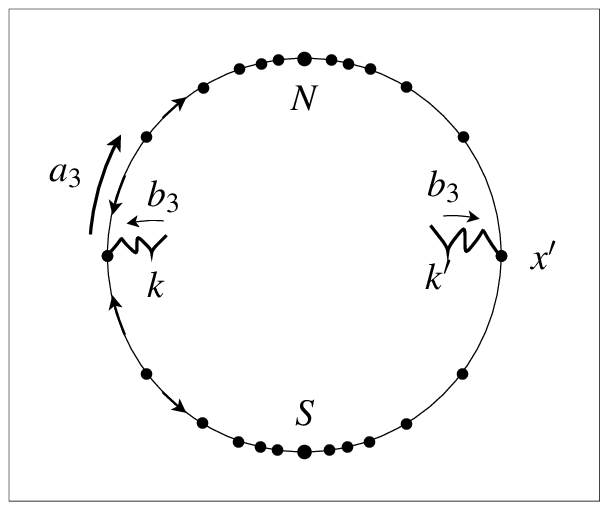}
\caption{Before and after the second prime-ends compactification}
\label{fig.before-after-two}
\end{center}
\end{figure}
 We have two homeomorphisms $a_{3},b_{3}$ of $\bbD^2$ that have no fixed point on $\inte(\bbD^2)$ and satisfy $a_{3}b_{3}a_{3}^{-1} = b_{3}^{-1}$.
 The map $a_{3}$ has exactly two fixed points which we denote by $N,S$. Each connected component $\Delta,\Delta'$ of $\partial \bbD^2 \setminus\{N,S\}$, 
 is preserved by $a_{3}, b_{3}$, and the action on $\Delta,\Delta'$ are conjugate to the action on figure~\ref{fig.dimension-1}.
There exists a compact connected set $k$, free for $a_{3}$, meeting $\Delta$ and such that $\lim_{a_{3}}^+ k = \Delta'$.
This set satisfies $b_{3}(k) \subset k$.

We apply Bikhoff's Lemma again to get
 a compact connected set $k'$  in $\bbD^2$, included in an arbitrarily small neighbourhood of   some fixed point $x' \in \Delta'$ for $b_{3}$, containing $x'$ but not included in $\Delta'$, and such that $b_{3}(k') \subset k'$. 

\begin{lemm*}
There exists $n_{0}$ such that for every $n \geq n_{0}$, $a_{3}^n(k)$ meets $k'$.
\end{lemm*}

\begin{proof}
This lemma is an immediate consequence of Proposition~\ref{prop.generalise55}, with $k^-=k$ and $k^+=k'$.
\end{proof}

Choose some even number $2n$ such that $a_{3}^{2n} k$ meets $k'$. Since $a_{3}^{2n}$ commutes with $b_{3}$, the set $k' = a_{3}^{2n} k$ again satisfies $b_{3}(k') \subset k'$.
By Brouwer's theory, for any point $z \in \inte(\bbD^2)$, the $\omega$-limit set $\omega(z)$ of $z$ for $b_{3}$ is included in $\partial \bbD^2$.
Since $b_{3}(k') \subset k'$,  if $z$ belongs to $k'$ then $\omega(z) \subset k' \cap \partial \bbD^2 \subset \Delta'$. Similarly if $z$ belongs to $a_{3}^{2n} k$ then 
$\omega(z) \subset k \cap \partial \bbD^2 \subset \Delta$. This contradicts the facts that 
$$
a_{3}^{2n} k \cap k' \neq \emptyset \mbox{ but } \Delta \cap \Delta' = \emptyset.
$$

\begin{figure}[htbp]
\begin{center}
\includegraphics{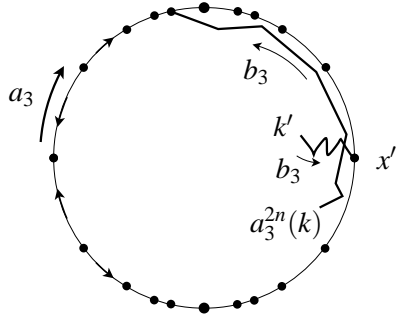}
\caption{The contradiction}
\label{fig.contradiction}
\end{center}
\end{figure}


\end{document}